\documentclass[12pt]{article}

\usepackage[a4paper]{geometry}

\usepackage[utf8]{inputenc}
\usepackage[english]{babel}

\usepackage{biblatex}
\usepackage{amsmath}
\usepackage{amsthm}
\usepackage{amssymb}

\newtheorem{theorem}{Theorem}
\newtheorem{lemma}{Lemma}
\newtheorem{definition}{Definition}
\newtheorem{corollary}{Corollary}
\newtheorem{remark}{Remark}

\title{Stability and Performance Analysis on Self-dual Cones}
\author{Emil Vladu\thanks{The author is with the Department of Automatic Control at Lund University. This work was partially supported by the Wallenberg AI, Autonomous Systems and Software Program (WASP) funded by the Knut and Alice Wallenberg Foundation.}}
\date{}
\begin{document}
\maketitle

\begin{abstract}
In this paper, we consider nonsymmetric solutions to certain Lyapunov and Riccati equations and inequalities with coefficient matrices corresponding to cone-preserving dynamical systems. Most results presented here appear to be novel even in the special case of positive systems. First, we provide a simple eigenvalue criterion for a Sylvester equation to admit a cone-preserving solution. For a single system preserving a self-dual cone, this reduces to stability. Further, we provide a set of conditions equivalent to testing a given H-infinity norm bound, as in the bounded real lemma. These feature the stability of a coefficient matrix similar to the Hamiltonian, a solution to two conic inequalities, and a stabilizing cone-preserving solution to a nonsymmetric Riccati equation. Finally, we show that the H-infinity norm is attained at zero frequency.
\end{abstract}

\section{Introduction} \label{sec:intro}
A monotone linear dynamical system is a system such that an input $u(t)$ and a state initial condition $x(0)$ confined to a cone together imply that the state $x(t)$ and output $z(t)$ are also confined to a cone for all $t \geq 0$. Although scarce and still at its infancy, research on the topic is warranted given the success of the special case in which the cones are taken as the nonnegative orthant. Such systems are known as positive systems and are particularly suited for analysis and synthesis of large-scale systems, e.g., \cite{farina2000positive}\cite{rantzer2018tutorial} and the references therein. One important reason is that $L_1$/$L_\infty$-gain verification and controller synthesis reduces to linear programming \cite{briat2013robust}\cite{rantzer2015scalable}. Similarly, positive diagonal solutions to Lyapunov equations and linear matrix inequalities (LMIs) are necessary and sufficient for achieving stability \cite{berman1994nonnegative} and a given $H_\infty$ norm bound \cite{tanaka2011bounded}, respectively. The latter then paves the way for structured synthesis in which the sparsity pattern on the controller may be specified in exchange for requiring closed-loop positivity \cite{tanaka2011bounded}. An additional reformulation of the Lyapunov theorem and the bounded real lemma given in \cite{ebihara2014lmi} is as follows: \textit{any} nonsymmetric solution with positive definite symmetric part is necessary and sufficient in order to certify stability and a particular $H_\infty$ norm, respectively. By contrast, in standard Lyapunov and $H_\infty$ theory, symmetric solutions to Lyapunov equations \cite{rugh1996linear} and LMIs \cite{boyd1994linear}\cite{gahinet1994lmi} or Riccati equations \cite{doyle1988state} are required to certify stability and performance, respectively.

For general proper cones, the key notions appear to be cone-preservance and cross-positivity, corresponding to nonnegative matrices and Metzler matrices, respectively, see Section \ref{sec:prel}. However, the above results do not generalize in a straightforward way to general cone-preserving/monotone systems. For example, extending the connection between the $L_1$ gain and linear programming in \cite{briat2013robust} requires a shift to cone linear absolute norms in \cite{shen2017input}. Further, \cite[Chapter 4]{tanaka2012symmetric} shows that the $H_\infty$ norm of a system which preserves a proper cone does not in general equal its static gain, as it does for positive systems \cite{rantzer2015scalable}. Instead, it is the spectral radius of a transfer function corresponding to such a system which achieves its maximum value at zero frequency \cite{tanaka2013dc}. Finally, the celebrated diagonal solution to Lyapunov equations and LMIs fails to hold more generally. However, if the solution is seen as the result of the quadratic representation of a Jordan algebra applied to a vector obtained through a conic program, then the symmetric cones, a subset of the self-dual ones, appear to be the natural setting for this property. This is indicated by recent works such as \cite{shen2016some} in the bounded real lemma case and \cite{lu2024kyp} in the KYP lemma case, corresponding to \cite{tanaka2011bounded} and \cite{rantzer2015kalman}, respectively, for positive systems. Note in particular that the $H_\infty$ norm is shown to equal the static gain in this setting \cite{shen2016some}. Another very recent result on symmetric cones is \cite{dalin2024special}, in which a Lie-algebraic approach is taken to construct a quadratic Lyapunov function for stable cone-preserving systems which becomes diagonal w.r.t. the nonnegative orthant.

By contrast, in this paper we explore what can be said when the proper cones are self-dual only. The flavor of our results is quite different from that of the above results, most notably in that we depart from matrix symmetry, let alone diagonality. In particular, we show that for two cross-positive matrices on a proper cone, the associated Sylvester equation admits a (possibly nonsymmetric) cone-preserving solution if and only if the sum of the greatest real parts of their eigenvalues is negative. It follows that when the cone is self-dual, a cone-preserving solution to the Lyapunov equation is necessary and sufficient for stability. Additionally, we leverage a recent result in \cite{vladu2024cone} to provide a set of conditions which are equivalent to a $\gamma$-bound on the $H_\infty$ norm for monotone systems w.r.t. self-dual cones: a) a given Riccati inequality has a solution with positive definite symmetric part, b) a stabilizing cone-preserving solution exists to a nonsymmetric Riccati equation, c) a particular coefficient matrix should be stable and d) a set of conic inequalities are satisfied by elements in the interior of the cone. Here, a) appears to generalize \cite{ebihara2014lmi} above, whereas b), c) and d) are novel to the best of the authors' knowledge, even for positive systems. For comparisons to similar results in the literature, see Remark \ref{rem:cf}. Finally, we show that the symmetry of the state cone in \cite{shen2016some} may in fact be relaxed to self-duality in order for the static gain to determine the $H_\infty$ norm.

The outline of the paper is as follows: Section \ref{sec:prel} reviews the basics of cone theory and cross-positivity in particular, and Section \ref{sec:results} presents the results of the paper. Section \ref{sec:example} illustrates the results with some examples and Section \ref{sec:proofs} provides the proofs. Finally, Section \ref{sec:conclusions} concludes the paper. 

\section{Preliminaries} \label{sec:prel}
In this section, we explain our notation and supply the required background for the contents in the remaining sections.

Let $\mathbb{R}$ denote the set of real numbers. $\mathbb{R}^n$ and $\mathbb{R}^{n \times m}$ denote the set of $n$-dimensional vectors and $n \times m$ matrices, respectively, with entries in $\mathbb{R}$. For $M \in \mathbb{R}^{n \times m}$, $\lVert M \rVert$ denotes the spectral (2-induced) norm of $M$, and $I$ is the identity matrix, with context determining its dimension. For $A \in \mathbb{R}^{n \times n}$, we say that $A$ is Hurwitz if all its eigenvalues have negative real part, and Metzler if all its offdiagonal elements are nonnegative. $\sigma (A)$ is the spectrum of $A$, i.e., the set of all its eigenvalues. If $A$ is symmetric, then $A \succ (\succeq) 0$ means that $A$ is positive (semi)definite.

We recall at this point the following important observation, made for instance in \cite{ebihara2014lmi}.
\begin{lemma} \label{lem:lyap_real}
    Let $A \in \mathbb{R}^{n \times n}$ and suppose that there exists a $P \in \mathbb{R}^{n \times n}$ with $P + P^T \succ 0$ such that $$A^T P^T + PA \prec 0.$$ If $\lambda$ is a real eigenvalue of $A$, then $\lambda < 0$.
\end{lemma}

A set $K \subseteq \mathbb{R}^n$ is called a cone if $x \in \mathcal{K}$ and $\alpha \geq 0$ imply $\alpha x \in K$. If in addition $K$ is convex, closed, pointed ($\mathcal{K} \cap \mathcal{-K} = \{0\}$) and has non-empty interior, then it is called a proper cone. A proper cone induces a partial order $\succeq_K$ on $\mathbb{R}^n$, i.e., $x \succeq_K y$ if and only if $x - y \in K$; strict inequality $x \succ_K y$ means that $x - y$ lies in the interior of $K$. The dual cone associated with a proper cone is $$K_* = \{y \mid y^T x \geq 0 \; \mathrm{for \; all} \; x \in K \}.$$ If $K = K_*$, then we say that $K$ is self-dual.

A matrix $A \in \mathbb{R}^{n \times n}$ is said to be cross-positive on $K$ if $x \in K$, $y \in K_*$ and $y^T x = 0$ imply $y^T A x \geq 0$. In particular, cross-positive matrices on the nonnegative orthant correspond to Metzler matrices. Another important notion associated with proper cones is that of $K$-nonnegativity: $A \in \mathbb{R}^{n \times n}$ is said to be $K$-nonnegative or preserve $K$ if $AK \subseteq K$. Since the set of $K$-nonnegative matrices is itself a proper cone in $\mathbb{R}^{n \times n}$ \cite[Chap. 1.1]{berman1994nonnegative}, it induces a partial order which we denote also by $\succeq_K$. Thus, $X \succeq_K Y$ for matrices $X, Y \in \mathbb{R}^{n \times n}$, as opposed to vectors, means that $X - Y$ is $K$-nonnegative. Similarly, $X \succ_K Y$ means that $X - Y$ maps all nonzero elements in $K$ into its interior, cf. \cite{schneider1970cross}. In the case that $K$ is the nonnegative orthant, $X \succeq_K 0$ and $X \succ_K 0$ mean that $X$ is entrywise nonnegative and positive, respectively. 

Cross-positive matrices satisfy the following property.
\begin{lemma} \label{lem:cross_krein} \cite[Theorem 5]{schneider1970cross}
    Let $K \subseteq \mathbb{R}^n$ be a proper cone and $A \in \mathbb{R}^{n \times n}$. Suppose now that $A$ is cross-positive on $K$. Then $\mu = \mathrm{max} \{\mathrm{Re}(\lambda ) \mid \lambda \in \sigma(A) \}$ is an eigenvalue of $A$. Further, $K$ contains an eigenvector corresponding to $\lambda$.
\end{lemma}

The cross-positivity of a block matrix is connected to that of its constituents in the following way.
\begin{lemma} \label{lem:kxk} \cite[Lemma 5]{vladu2024cone}
    Let the proper cone $K \subseteq \mathbb{R}^n$ and the matrices $A, B, C, D$ $\in \mathbb{R}^{n \times n}$ be given. Then $$L = \begin{pmatrix} A && B \\ C && D \end{pmatrix}$$ is cross-positive on $K \times K$ if and only if $A$ and $D$ are cross-positive on $K$ and $B, C \succeq_K 0$.
\end{lemma}

An important connection between dynamical systems and the notion of cross-positivity is given by the following lemma.
\begin{lemma} \label{lem:cross_eAt} \cite[Theorem 3]{schneider1970cross} Let $K \subseteq \mathbb{R}^n$ be a proper cone and $A \in \mathbb{R}^{n \times n}$. Then $A$ is cross-positive on $K$ if and only if $e^{At}$ is $K$-nonnegative for all $t \geq 0$.
\end{lemma}

It follows that for a system $\dot{x} = Ax$ with $A$ cross-positive on $K$, the resulting trajectory $x(t)$ given $x(0) \in K$ will remain inside $K$ for all $t \geq 0$.

We have the following stability test for cross-positive matrices, see e.g., \cite{shen2016some}.
\begin{lemma} \label{lem:cross_stable} \cite[Facts 7.1, 7.3, 7.5]{schneider2006matrices} Suppose $A  \in \mathbb{R}^{n \times n}$ is cross-positive on a proper cone $K \subseteq \mathbb{R}^n$. Then the following are equivalent:
\begin{enumerate}
    \item[(i)] $A$ is stable.
    \item[(ii)] There exists $x \succ_K 0$ such that $Ax \prec_K 0$.
    \item[(iii)] $A$ is invertible and $-A^{-1}$ is $K$-nonnegative.
\end{enumerate}
\end{lemma}

A recent result characterizes the existence of a stabilizing $K$-nonnegative solution to a nonsymmetric Algebraic Riccati Equation.
\begin{lemma} \label{lem:emil} \cite[Theorem 1]{vladu2024cone}     
Let the proper cone $K \subseteq \mathbb{R}^n$ and the matrices $A, B, C, D \in \mathbb{R}^{n \times n}$ be given. Suppose now that $$L = \begin{pmatrix} A && B \\ C && D \end{pmatrix}$$ is cross-positive on $K \times K$. Then $L$ is stable if and only if
    \begin{equation*}
    XBX + DX + XA + C = 0
    \end{equation*} has a solution $X_* \succeq_K 0$ such that $A + BX_*$ and $D + X_* B$ are Hurwitz and cross-positive on K.
\end{lemma}

We close this section by considering linear time-invariant (LTI) systems with state matrix $A \in \mathbb{R}^{n \times n}$, input matrix $B \in \mathbb{R}^{n \times m}$ and output matrix $C \in \mathbb{R}^{p \times n}$. We consider only the case with no direct term. Recall now that for such stable systems with transfer function $G(s) = C ( sI - A) ^{-1} B$, the $H_\infty$ norm is defined as $\lVert G \rVert _\infty = \sup_{\omega} \lVert G(i \omega ) \rVert$. A variant of the well-known bounded real lemma is as follows.
\begin{lemma} \label{lem:kyp} \cite[Corollary 12.3]{zhou1998essentials}
    Let $\gamma > 0$ and suppose that $A$ is Hurwitz. Define now $$H = \begin{pmatrix} A && \frac{1}{\gamma ^{2}} BB^T \\ -C^T C && -A^T \end{pmatrix}.$$ The following conditions are equivalent:
    \begin{enumerate}
        \item[(i)] $\lVert G \rVert _\infty < \gamma$
        \item[(ii)] $H$ has no eigenvalues on the imaginary axis.
        \item[(iii)] There exists a $P \succeq 0$ such that \begin{equation} \label{eq:ric_eq_normal} \frac{1}{\gamma^{2}} PBB^T P + A^T P + PA + C^T C = 0 \end{equation} and $A + \frac{1}{\gamma ^2}BB^T P$ has no imaginary axis eigenvalues.
        \item[(iv)] There exists a $P \succ 0$ such that \begin{equation} \label{eq:ric_normal} \frac{1}{\gamma^{2}} PBB^T P + A^T P + PA + C^T C \prec 0. \end{equation}
    \end{enumerate}
\end{lemma}

\section{Results} \label{sec:results}
In this section, we present the results of the paper. Those related to stability are found in Subsection \ref{subsec:stability}, whereas those related to performance are found in Subsection \ref{subsec:gain}.

\subsection{Stability Analysis} \label{subsec:stability}
Define $\mu (A)$ to be the greatest real part of the spectrum of a square matrix $A$. We then have the following result.
\begin{theorem} \label{thrm:sylvester}
    Suppose $A, D \in \mathbb{R}^{n \times n}$ are cross-positive on a proper cone $K$. Then there exists $P \succ_K 0$ such that $$DP + PA \prec_K 0$$ if and only if $$\mu (A) + \mu (D) < 0.$$
\end{theorem}
\begin{proof}
    See Section \ref{sec:proofs}.
\end{proof}
For self-dual cones in particular, this reduces to a stability test.
\begin{corollary} \label{cor:lyapunov}
    Suppose $A \in \mathbb{R}^{n \times n}$ is cross-positive on a self-dual proper cone $K$. Then there exists $P \succ_K 0$ such that $$A^T P + PA \prec_K 0$$ if and only if $A$ is Hurwitz.
\end{corollary}
\begin{proof}
    See Section \ref{sec:proofs}.
\end{proof}
We close this subsection with some remarks.
\begin{remark}
    Corollary \ref{cor:lyapunov} should be compared to Lyapunov's Theorem, e.g., \cite{rugh1996linear}. Note, however, that compared to this standard result, the solution $P$ in Corollary \ref{cor:lyapunov} does not have to be symmetric. For example, when $K$ is taken as the nonnegative orthant so that cross-positivity reduces to the Metzler property, $P \succ_K 0$ is equivalent to $P > 0$, i.e., entrywise positivity. 
\end{remark}
\begin{remark}
    In this remark, we give a Lyapunov function-like interpretation of Theorem \ref{thrm:sylvester}. Given the two systems $\dot{x} = Ax$ and $\dot{y} = D^T y$ with $A$ and $D$ cross-positive on $K$, clearly there exists a solution $P \succ_K 0$ if and only if there exists a quadratic function $V(x, y) = y^T P x$ such that $V(x, y) > 0$ and $\dot{V}(x, y) = y^T (D P + PA) x < 0$ for all nonzero $x \in K$ and $y \in K_*$. Thus, for nonzero trajectories in $K$ and $K_*$, respectively, $V(x, y)$ is always positive and decreasing. Of course, it is intuitive that this can happen despite one of the systems being unstable, so long as the trajectories of the other system converge faster to the origin: this is consistent with the condition $\mu(A) + \mu(D) < 0$.
\end{remark}

\subsection{Performance Analysis} \label{subsec:gain}
In this subsection, we consider the LTI system
\begin{equation} \label{eq:sys}
\begin{aligned}
    \dot{x} &= Ax + B u \\
    z &= Cx
\end{aligned}
\end{equation}
where $x(t) \in \mathbb{R}^n$ is the state, $u(t) \in \mathbb{R}^{m}$ is the control input and $z(t) \in \mathbb{R}^p$ is the regulated output. Denote by $G$ the open-loop transfer function from $u$ to $z$.

In keeping with \cite{angeli2003monotone} and \cite{shen2017input}, we make the following definition.
\begin{definition} \label{def:cones}
    System (\ref{eq:sys}) is said to be monotone with respect to the proper cones $( K_u, K_x, K_z ) \subseteq ( \mathbb{R}^m , \mathbb{R}^n , \mathbb{R}^p )$ if $u(t) \in K_u$ for all $t \geq 0$ and $x(0) \in K_x$ together imply that $x(t) \in K_x$ and $z(t) \in K_z$ for all $t \geq 0$. 
\end{definition}

The main result of this subsection is the following.
\begin{theorem} \label{thrm:kyp_ext}
    Consider system (\ref{eq:sys}) with $A$ Hurwitz and let $\gamma > 0$ and the three self-dual proper cones $( K_u, K_x, K_z ) \subseteq ( \mathbb{R}^m , \mathbb{R}^n , \mathbb{R}^p )$ be given. Suppose now that (\ref{eq:sys}) is monotone with respect to $( K_u, K_x, K_z )$. Then the following conditions are equivalent:
    \begin{enumerate}
        \item[(i)] $\lVert G \rVert _\infty < \gamma$.
        \item[(ii)] $L$ is Hurwitz, where $$L = \begin{pmatrix}
            A & \frac{1}{\gamma^2}BB^T \\
            C^T C & A^T
        \end{pmatrix}.$$
        \item[(iii)] There exists $P \succeq_{K_x} 0$ such that 
        \begin{equation} \label{eq:ric} 
        \frac{1}{\gamma^{2}}PB B^T P + A^T P + PA + C^T C = 0 
        \end{equation} 
        and $A + \frac{1}{\gamma ^2}BB^TP$ is Hurwitz.
        \item[(iv)] There exists $P \in \mathbb{R}^{n \times n}$ with $P + P^T \succ 0$ such that \begin{equation} \label{eq:ric_extended} \frac{1}{\gamma^{2}} PB B^T P^T + A^T P^T + PA + C^T C \prec 0. \end{equation}
        \item[(v)] There exists $p, q \succ_{K_x} 0$ such that
        \begin{equation*}
        \begin{aligned}
            Ap + \frac{1}{\gamma ^2} BB^T q &\prec_{K_x} 0 \\
            C^T C p + A^T q &\prec_{K_x} 0.
        \end{aligned}            
        \end{equation*}
    \end{enumerate}
\end{theorem}
\begin{proof}
    See Section \ref{sec:proofs}.
\end{proof}
The next result shows that monotone systems w.r.t. self-dual cones also achieve their $H_\infty$ norm at zero frequency.
\begin{theorem} \label{thrm:gain}
    Consider system (\ref{eq:sys}) with $A$ Hurwitz and suppose that system (\ref{eq:sys}) is monotone with respect to the self-dual proper cones $( K_u, K_x, K_z ) \subseteq ( \mathbb{R}^m , \mathbb{R}^n , \mathbb{R}^p )$. Then $$\lVert G \rVert _\infty = \lVert G(0) \rVert .$$
\end{theorem}
\begin{proof}
    See Section \ref{sec:proofs}.
\end{proof}
We close this subsection with some remarks.
\begin{remark} \label{rem:cone_def}
    It is straightforward to verify using Lemma \ref{lem:cross_eAt} that Definition \ref{def:cones} is equivalent to the fact that $A$ is cross-positive on $K_x$, $BK_u \subseteq K_x$ and $CK_x \subseteq K_z$, see e.g., \cite{shen2017input}. Further, when the three cones are taken as the nonnegative orthant, we regain the definition of an (internally) positive system, e.g., \cite{rantzer2018tutorial}.
\end{remark}
\begin{remark} \label{rem:cf}
    Conditions (ii) and (iii) in Theorem \ref{thrm:kyp_ext} appear to be analogous to the imaginary axis eigenvalue condition on the Hamiltonian and the Riccati equation solution condition in the standard bounded real lemma, respectively, cited here for convenience in Lemma \ref{lem:kyp}. Condition (v) is perhaps best compared to condition (ii) in \cite[Theorem 2]{shen2016some} in which twice the amount of variables and inequalities are used to characterize the $\gamma$-suboptimality of the static gain for symmetric cones.
\end{remark}

\section{Illustrative Examples} \label{sec:example}
In this section, we provide examples of some of the results in Section \ref{sec:results}. For the purpose of illustration, we shall consider only the nonnegative orthant; the reader is referred to the references provided in Section \ref{sec:intro} for the many examples and applications on alternative cones.

\subsection{A closed-loop positive $H_\infty$ Optimal Controller} \label{subsec:example_irrigation}
In this subsection, we consider the dynamics
$$\dot{x} = \begin{pmatrix} -1 && 0 && 0 \\ 1 && -2 && 0 \\ 0 && 2 && -4 \end{pmatrix} x + \begin{pmatrix} -1 && 0 \\ 1 && -1 \\ 0 && 1 \end{pmatrix}u + w.$$
The above system can be thought of as a simple model of a small irrigation network consisting of three pools connected in series, e.g., \cite{cantoni2007control}. The state matrix $A$ suggests that each pool decays towards some equilibrium level, and that the decaying content does not vanish but is instead transferred over to the next pool. Further, the input matrix $B$ suggests that we may actuate a transfer of contents between two adjacent pools. Finally, each pool is subject to disturbing inflows $w$.

Now, it was shown in \cite[Theorem 2]{vladu2022decentralized} that the controller $$K_* = B^T A^{-T} = \begin{pmatrix} 1 & 0 & 0 \\ 0 & \frac{1}{2} & 0 \end{pmatrix}$$ results in the following three desirable properties for the closed-loop system from $w$ to $(x, u)$:
\begin{enumerate}
    \item[a)] $K_*$ is $H_\infty$ optimal.
    \item[b)] $K_*$ is closed-loop positive.
    \item[c)] $K_*$ is diagonal.
\end{enumerate}
In a sense, $K_*$ arguably appears to be a very natural candidate controller. However, this controller cannot in fact result from the $H_\infty$ synthesis Riccati inequality \cite[Section 7.5.1]{boyd1994linear} \begin{equation} \label{eq:ric_p_sym}
    A^T P + P A + P \big (\frac{1}{\gamma ^2} I - BB^T \big) P + C^T C \prec 0,
\end{equation}
as there is in fact no symmetric $P$ such that $K_* = -B^T P$. In order to see this, suppose on the contrary that such a $$P = P^T = \begin{pmatrix} p_{11} & p_{12} & p_{13} \\ p_{12} & p_{22} & p_{23} \\ p_{13} & p_{23} & p_{33} \end{pmatrix}$$ existed. We would then have $$\begin{pmatrix} 1 & 0 & 0 \\ 0 & \frac{1}{2} & 0 \end{pmatrix} = \begin{pmatrix} p_{11} - p_{12} & p_{12} - p_{22} & p_{13} - p_{23} \\ p_{12} - p_{13} & p_{22} - p_{23} & p_{23} - p_{33} \end{pmatrix}.$$ But this would imply that $p_{22} = p_{12} = p_{13} = p_{23}$ so that $0 = p_{22} - p_{23} = \frac{1}{2}$, a contradiction.

On the other hand, matters become different once we consider the extended version $$A^T P^T + P A + P \big (\frac{1}{\gamma ^2} I - BB^T \big) P^T + C^T C \prec 0.$$ It is straightforward to show that the nonsymmetric matrix $P_* = -A^{-1}$ is a solution for all $\gamma > \lVert (AA^T + BB^T )^{-1} \rVert ^\frac{1}{2}$, where the latter value is known to be a lower bound over all stabilizing controllers. Thus, invoking condition (iv) in Theorem \ref{thrm:kyp_ext} for the corresponding closed-loop system, $K_* = -B^T P_*^T = B^T A^{-T}$ is seen to be optimal.

A synthesis procedure like (\ref{eq:ric_p_sym}) which fails to account for a natural controller such as $K_*$ is arguably incomplete. Of course, any stabilizing $\gamma$-suboptimal controller can still be reached by performing a variable change and searching over two matrix variables instead of one in the corresponding LMI \cite{boyd1994linear}. However, this example shows that there is substance in the middle ground offered by the above nonsymmetric extension, and searching over one variable instead of two may prove valuable in the context of large-scale systems.

\subsection{A Positive Solution to the Sylvester Equation} \label{subsec:example_sylvester}
In this subsection, we illustrate Theorem \ref{thrm:sylvester} on the nonnegative orthant by considering the two matrices $$A = \begin{pmatrix}
    -2 & 1 \\ 0 & -2
\end{pmatrix}, \; \; \; \; D = \begin{pmatrix}
    -1 & 0 \\ 4 & 1 
\end{pmatrix}$$ which are clearly Metzler and thus cross-positive on the nonnegative orthant. Since further their eigenvalues lie on the diagonal, it is clear that $\mu (A) + \mu (D) = -2 + 1 = -1 < 0$ so that Theorem \ref{thrm:sylvester} gives the existence of a $P > 0$ such that $DP + PA < 0$. One such $P$ is given by $$P = \frac{1}{9} \begin{pmatrix}
    3 & 4 \\ 21 & 46
\end{pmatrix} > 0$$ since then $$DP + PA = \begin{pmatrix}
    -1 & -1 \\ -1 & -1
\end{pmatrix} < 0.$$

\section{Proofs} \label{sec:proofs}
In this section, we supply proofs to the results in Section \ref{sec:results}. For this purpose, we shall require the following lemma.

\begin{lemma} \label{lem:gain}
    Suppose $\gamma > 0$ and the matrices $A \in \mathbb{R}^{n \times n}$, $B \in \mathbb{R}^{n \times m}$ and $C \in \mathbb{R}^{p \times n}$ are given with $A$ invertible. If $\lambda$ is an eigenvalue of $$H = \begin{pmatrix} A && \frac{1}{\gamma ^{2}} BB^T \\ -C^T C && -A^T \end{pmatrix},$$ then $\lambda \ne 0$ for all $\gamma > \lVert CA^{-1} B \rVert$. Further, if $\gamma = \lVert CA^{-1} B \rVert$, then $H$ has a zero eigenvalue.
\end{lemma}
\begin{proof}
    We have
    \begin{equation*}
    \begin{aligned}
        \mathrm{det}(H) &= \mathrm{det}(A)\mathrm{det}(-A^T - (-C^T C) A^{-1} (\gamma^{-2} BB^T )) \\
        &= \mathrm{det}(A) \mathrm{det}(-A^T ) \\ & \; \; \; \; \; \cdot \mathrm{det}(I - \gamma^{-2} A^{-T}C^T C A^{-1} BB^T ) \\ 
        &= \mathrm{det}(A) \mathrm{det}(-A^T) \\ & \; \; \; \; \; \cdot \mathrm{det}(I - \gamma^{-2} (B^T A^{-T}C^T) (C A^{-1} B )) \\
        &= \mathrm{det}(A) \mathrm{det}(-A^T) \mathrm{det}(I - \gamma ^{-2}Q^T Q)
    \end{aligned}
    \end{equation*}
    where $Q = CA^{-1} B$. Here, the first equality follows from the Schur complement determinant rule \cite[Ch. 0.8.5]{horn2012matrix}, the second equality from the standard product determinant rule and the third equality from a repeated application of the Schur complement determinant rule so that $\mathrm{det}(I + RS) = \mathrm{det}(I + SR)$ with $R \in \mathbb{R}^{r \times s}$ and $S \in \mathbb{R}^{s \times r}$. But $\lVert CA^{-1} B \rVert$ is exactly the square root of the largest eigenvalue of $Q^T Q$, a matrix with nonnegative eigenvalues. It follows that the eigenvalues $\lambda$ of $\gamma^{-2} Q^T Q$ must satisfy $0 \leq \lambda < 1$ for all $\gamma > \lVert CA^{-1} B \rVert$. Thus, $\mathrm{det}(I - \gamma^{-2} Q^T Q ) \ne 0$ for such $\gamma$, exploiting the fact that the determinant is equal to the product of the eigenvalues. Further, since $A$ is invertible and hence $\mathrm{det}(A) = \mathrm{det}(A^T) \ne 0$, it follows that $\mathrm{det}(H) \ne 0$. Thus, no eigenvalue of $H$ can be zero for $\gamma > \lVert CA^{-1} B \rVert $, or else the eigenvalue product would be zero. If on the other hand $\gamma = \lVert CA^{-1} B \rVert$, then $\gamma^{-2} Q^T Q$ will have an eigenvalue $\lambda = 1$. As a result, similar reasoning gives $\mathrm{det}(H) = 0$, i.e., $H$ must have a zero eigenvalue.
\end{proof}

\begin{proof} \textbf{Theorem \ref{thrm:sylvester}}
    

    \noindent $\Leftarrow$: Recall first that for any $M \in \mathbb{R}^{n \times n}$ with Jordan decomposition $M = SJS^{-1}$, we have $e^{Mt} = Se^{Jt}S^{-1} = S\Bar{D}(t)\Bar{E}(t)S^{-1}$. Here, $\Bar{D}(t)$ is diagonal and consists of entries $e^{\lambda t}$ with $\lambda \in \sigma (M)$, and $\Bar{E}(t)$ has polynomial entries in $t$. As such, for any $Q \prec_K 0$,
    \begin{equation*}
    P = \int_0^\infty e^{D t}(-Q)e^{A t}\mathrm{dt}
    \end{equation*}
    must converge. This follows due to the assumption $\mu (A) + \mu (D) < 0$, as each entry in the integrand is a sum of terms with factors $e^{({\lambda_D}_i + {\lambda_A}_j)t}$. In particular, $e^{D t}Qe^{A t} \rightarrow 0$ as $t \rightarrow \infty$ so that
    \begin{equation*}
    \begin{aligned}
    Q &= \big [ e^{Dt}(-Q)e^{At} \big ] _0^\infty = \int_0^\infty \frac{d}{dt} \big( e^{Dt}(-Q)e^{At} \big ) \mathrm{dt} \\ &= \int_0^\infty \big ( De^{Dt}(-Q)e^{At} + e^{Dt}(-Q)e^{At}A \big ) \mathrm{dt} = DP + PA.
    \end{aligned}
    \end{equation*}
    Finally, Lemma \ref{lem:cross_eAt} implies that $P \succeq_K 0$, as $-Q \succ_K 0$ and the set of $K$-nonnegative matrices is closed as it is a proper cone, see Section \ref{sec:prel}. We may now lift $P$ into the interior by adding a sufficiently small perturbation.
    
    \noindent $\Rightarrow$: Because $A$ is cross-positive by assumption, by Lemma \ref{lem:cross_krein} there is a $v \succeq_K 0$ with $v \ne 0$ such that $Av = \mu (A) v$. Thus, \begin{equation*} 0 \succ_K Qv = (DP + PA)v = DPv + \mu (A) Pv = (D + \mu (A) I) Pv. \end{equation*} Since $Pv \succ_K 0$ and $D$ is cross-positive, Lemma \ref{lem:cross_stable} shows that $D + \mu (A) I$ is stable and the conclusion follows.
\end{proof}

\begin{proof} \textbf{Corollary \ref{cor:lyapunov}}

    \noindent This follows by invoking Theorem \ref{thrm:sylvester} after noting that $A$ is Hurwitz if and only if $\mu (A) = \mu (A^T ) < 0$, and that $A^T$ is also cross-positive on $K$ if $A$ is, provided that $K$ is self-dual. The latter follows by noting that $K_{**} = K$, see e.g., \cite[p. 53]{boyd2004convex}, so that $x \in K_{**} = K$, $y \in K_*$ and $x^T y = 0$ imply $x^T A^T y = (x^T A^T y)^T = y^T A x \geq 0$. 

\end{proof}

\begin{proof} \textbf{Theorem \ref{thrm:kyp_ext}}

    \noindent We prove the following chain of implications: $(i) \Rightarrow (iv) \Rightarrow (ii) \Rightarrow (iii) \Rightarrow (i)$. We subsequently show that $(ii) \Leftrightarrow (v)$.
    
    \underline{$(i) \Rightarrow (iv)$}: This follows immediately from Lemma \ref{lem:kyp}.

    \underline{$(iv) \Rightarrow (ii)$}: We have
    \begin{equation*}
        \begin{aligned}
           &\begin{pmatrix}
                I && 0 \\ 0 && P
            \end{pmatrix}
           \begin{pmatrix}
                A && \frac{1}{\gamma ^2}BB^T \\ C^T C && A^T
            \end{pmatrix} ^T
            \begin{pmatrix}
                P^T && 0 \\ 0 && I 
            \end{pmatrix}
            + \\ &
            \begin{pmatrix}
                P && 0 \\ 0 && I 
            \end{pmatrix}
            \begin{pmatrix}
                A && \frac{1}{\gamma ^2} BB^T \\ C^T C && A^T
            \end{pmatrix}
            \begin{pmatrix}
                I && 0 \\ 0 && P^T 
            \end{pmatrix}
            = \\ &
            \begin{pmatrix}
                A^T P^T + PA && \frac{1}{\gamma ^2} PBB^T P^T + C^T C \\ \frac{1}{\gamma ^2} PBB^T P^T + C^T C && A^T P^T + PA
            \end{pmatrix}
            = \\ &
            \begin{pmatrix}
                R - F && F \\ F && R - F
            \end{pmatrix}
            =
            \begin{pmatrix}
                R && 0 \\ 0 && R
            \end{pmatrix}
            +
            \begin{pmatrix}
                -F && F \\ F && -F
            \end{pmatrix}
            \prec 0
        \end{aligned}
    \end{equation*}
    where $R = \frac{1}{\gamma ^2} PBB^T P^T + A^T P^T + PA + C^T C$ and $F = \frac{1}{\gamma ^2} PBB^T P^T + C^T C$. In order to see why the above expression is negative definite, note that if $x = (y ^T , z ^T )^T \in \mathbb{R}^{2n}$, we have $$x^T \begin{pmatrix}
        -F && F \\ F && -F
    \end{pmatrix} x
    =
    -(y - z)^T F(y - z) \leq 0$$ since clearly $F \succeq 0$. Thus, negative definiteness follows as $R \prec 0$ by assumption. A congruence transformation now gives
    $$L^T \begin{pmatrix} P^T && 0 \\ 0 && P^{-T} \end{pmatrix} + \begin{pmatrix} P && 0 \\ 0 && P^{-1} \end{pmatrix} L \prec 0$$ and since $P + P^T \succ 0$ by assumption, another congruence transformation gives $P^{-T}(P^T + P)P^{-1} = P^{-1} + P^{-T} \succ 0$ so that for any real eigenvalue $\lambda$ of $L$, $\lambda < 0$ by Lemma \ref{lem:lyap_real}.

    In order to show that this implies stability, invoke the assumption of monotonicity through Remark \ref{rem:cone_def} to see that $A$ and $A^T$ are both cross-positive on $K_x$, see the proof of Corollary \ref{cor:lyapunov}. Further, $B K_u \subseteq K_x$ and $CK_x \subseteq K_z$ clearly imply $B ^T K_{x_*} \subseteq K_{u_*}$ and $C^T K_{z_*} \subseteq K_{x_*}$, respectively. It follows, since $K_u$, $K_x$ and $K_z$ are self-dual by assumption, that $B ^T K_{x} \subseteq K_{u}$ and $C^T K_{z} \subseteq K_{x}$. Thus, $B B^T \succeq_{K_x} 0$ and $C^T C \succeq_{K_x} 0$. Consequently, Lemma \ref{lem:kxk} gives that $L$ is cross-positive on $K \times K$, and it follows from Lemma \ref{lem:cross_krein} that $L$ has a real eigenvalue which upper bounds the real part of any other eigenvalue. But since all real eigenvalues of $L$ must be negative according to the above, $L$ must be Hurwitz.

    \underline{$(ii) \Rightarrow (iii)$}: This follows immediately from Lemma \ref{lem:emil}.

    \underline{$(iii) \Rightarrow (i)$}: Apply the following well-known similarity transformation to $H$:
    \begin{equation*}
        \begin{aligned}
           &\begin{pmatrix}
                I && 0 \\ -P && I 
            \end{pmatrix}
           \begin{pmatrix} A && \frac{1}{\gamma ^2}BB^T \\ -C^T C && -A^T \end{pmatrix}
            \begin{pmatrix}
                I && 0 \\ P && I 
            \end{pmatrix}
            = \\ &
            \begin{pmatrix}
                A + \frac{1}{\gamma ^2}BB^T P && \frac{1}{\gamma ^2}BB^T \\ 0 && -(A + \frac{1}{\gamma ^2}BB^T P ^T )^T
            \end{pmatrix}
        \end{aligned}
    \end{equation*}
    Here, the fact that $P$ solves (\ref{eq:ric}) was exploited to obtain the zero block in the last expression. Now, this final matrix is block triangular, and so its eigenvalues coincide with those of the matrices on the diagonal, i.e., $\sigma (A + \frac{1}{\gamma ^2}BB^T P ) \subseteq \sigma (H)$. Since $A + \frac{1}{\gamma ^2}BB^T P $ is Hurwitz by assumption and contains $n$ eigenvalues, the remaining $n$ eigenvalues must have positive real part due to the symmetric eigenvalue distribution of Hamiltonians about the imaginary axis \cite[p. 233]{zhou1998essentials}. It follows that $H$ can have no eigenvalues on the imaginary axis, and condition (i) thus follows from Lemma \ref{lem:kyp}.

    \underline{$(ii) \Leftrightarrow (v)$}: This follows directly from Lemma \ref{lem:cross_stable}.
\end{proof}

\begin{proof} \textbf{Theorem \ref{thrm:gain}}
    
    \noindent Suppose on the contrary that $\lVert G \rVert _\infty > \lVert G(0) \rVert$, i.e., there exists a $\gamma_* > 0$ such that $\lVert G(0) \rVert < \gamma_* < \lVert G \rVert _\infty$. By Theorem \ref{thrm:kyp_ext}, this means that $L( \gamma_* )$ cannot be Hurwitz, where $$L(\gamma ) = \begin{pmatrix}
            A & \frac{1}{\gamma^2}BB^T \\
            C^T C & A^T
    \end{pmatrix}.$$ At the same time, by the continuity of eigenvalues, $L$ must be Hurwitz for some sufficiently large $\gamma$, say $\gamma_+$, as $A$ is Hurwitz by assumption. But since $L$ is cross-positive on $K \times K$ (see the proof of Theorem \ref{thrm:kyp_ext}), Lemma \ref{lem:cross_krein} implies that $L(\gamma )$ has a real eigenvalue $\lambda (\gamma )$ with maximal real part over its spectrum for all $\gamma_* \leq \gamma \leq \gamma_+$. Now, since $\lambda (\gamma_* ) \geq 0$ and $\lambda (\gamma_+ ) < 0$, again by the continuity of eigenvalues there must exist some $\gamma_0 \geq \gamma_*$ such that $\lambda (\gamma_0 ) = 0$, i.e., $$H(\gamma_0 ) = \begin{pmatrix}
            A & \frac{1}{\gamma_0^2}BB^T \\
            -C^T C & -A^T
    \end{pmatrix} = \begin{pmatrix}
            I & 0 \\
            0 & -I
    \end{pmatrix} L(\gamma_0 )$$ has zero eigenvalue. But this is a contradiction by Lemma \ref{lem:gain} as $\lVert CA^{-1}B \rVert = \lVert G(0) \rVert < \gamma_* \leq \gamma_0$. Thus, $\lVert G \rVert _\infty \leq \lVert G(0) \rVert$, and so $$\lVert G(0) \rVert \leq \sup_{\omega} \lVert G(i \omega ) \rVert = \lVert G \rVert _\infty \leq \lVert G(0) \rVert$$ and the conclusion follows.
\end{proof}

\section{Conclusions} \label{sec:conclusions}
In this paper, we have explored cone-preserving (possibly nonsymmetric) solutions to Lyapunov and Riccati equations and inequalities for the purpose of stability and performance verification. Overall, these results apply to LTI systems with zero direct term that are monotone w.r.t. self-dual cones, and most of the results appear to be novel also in the special case of positive systems. Given the latter's success, the main purpose of the paper has been to complement it with a new angle as well as to pin down the structure it hinges on: the self-duality of the nonnegative orthant. By contrast, previous results indicate that the celebrated diagonal solution property and its many consequences are generated by the symmetric cones, a subset of the self-dual ones. One concrete benefit of this distinction is that we may now attribute properties such as the $H_\infty$ norm being determined by the static gain to the former structure rather than the latter, instead of collapsing the two.

Although the main contribution of the present paper is arguably one of understanding, it also hints at potential usage in areas such as controller synthesis. In particular, synthesis based on diagonal solutions as in \cite{tanaka2011bounded} gives $H_\infty$ optimality only for the restricted set of stabilizing closed-loop positive controllers. In a worst-case scenario, this optimal value may be very far from the optimal value over the entire set of stabilizing controllers. By contrast, in Subsection \ref{subsec:example_irrigation}, we see how closed-loop positive controllers are featured quite naturally as optimal also within this wider framework.

\section{Acknowledgements}
The author is grateful to Prof. Anders Rantzer and Dr. Dongjun Wu for proofreading the manuscript and providing helpful comments.

\printbibliography

\end{document}